\documentclass[12pt, reqno]{amsart}
\usepackage{amsmath, amsthm, amscd, amsfonts, amssymb, graphicx, color}
\usepackage{setspace}
\usepackage{mathrsfs}
\usepackage{multicol}
\usepackage[bookmarksnumbered, colorlinks, plainpages]{hyperref}
\hypersetup{colorlinks=true,linkcolor=black, anchorcolor=green, citecolor=cyan, urlcolor=black, filecolor=magenta, pdftoolbar=true}

\newtheorem{theorem}{Theorem}[section]
\newtheorem{lemma}[theorem]{Lemma}
\newtheorem{proposition}[theorem]{Proposition}
\newtheorem{corollary}[theorem]{Corollary}
\theoremstyle{definition}

\theoremstyle{remark}

\numberwithin{equation}{section}

\newcommand{\NN}{\mathbb{N}}

\newcommand{\CC}{\mathbb {C}}

\begin{document}
\setcounter{page}{1}
\title[Spectrums  and uniform mean ergodicity ]{Spectrums and uniform  mean ergodicity of weighted composition operators on Fock spaces }
\author [Werkaferahu Seyoum ]{Werkaferahu Seyoum}
\address{Department of Mathematics,
Kotebe Metropolitan University, Ethiopia}
\email{Werkaferahus@gmail.com}
\author[Tesfa  Mengestie*]{Tesfa  Mengestie* }\footnote{*Corresponding author}
\address{Mathematics section \\
Western Norway University of Applied Sciences\\
Klingenbergvegen 8, N-5414 Stord, Norway}
\email{Tesfa.Mengestie@hvl.no}


 \begin{abstract}
 For  holomorphic pairs of symbols $(u, \psi)$, we  study  various  structures of  the weighted composition operator $ W_{(u,\psi)} f= u \cdot f(\psi)$ defined
 on the Fock spaces $\mathcal{F}_p$.   We have   identified operators    $W_{(u,\psi)}$ that have  power bounded and uniformly mean ergodic properties on the spaces. These properties are described in terms of easy to apply conditions  relying on  the values   $|u(0)|$  and  $|u(\frac{b}{1-a})|$ where  $ a$ and $b$ are coefficients    from   linear expansion of the symbol  $\psi$. The spectrum of the operators are also  determined  and applied  further   to prove results about    uniform mean ergodicity.
\end{abstract}
\maketitle
\section{Introduction}\label{sec1}
We denote by $\mathcal{H}(\CC)$ the  space of  analytic functions on  the complex plane $\CC$.  For pairs of functions $(u, \psi)$  in $\mathcal{H}(\CC)$,  the weighted composition operator $W_{(u,\psi)}$   is defined    by $W_{(u,\psi)}f= u\cdot  f(\psi), \   f\in \mathcal{H}(\CC)$. The operator  generalizes both the composition $C_\psi$ and multiplication $M_u$ operators since it can be factored as  $W_{(u,\psi)}= M_u C_\psi $. Weighted composition operators have been a subject of intense studies in the last several years partly because they  found applications in the description of isometries on spaces of analytic functions; see the  monographs  \cite{FJ, FJ1} for detailed accounts.  For studies on the various properties of
the operators, for example on the classical Fock spaces $\mathcal{F}_p$,  one may consult the works in \cite{Tle, TM4, TM0, UK} and the references therein. Recall that $\mathcal{F}_p$  are   spaces  consisting of all entire functions $f$ for which
 \begin{align*}
 \|f\|_{ p}= \begin{cases}\Big(\frac{p}{2\pi} \int_{\CC} |f(z)|^p
e^{-\frac{p|z|^2}{2}}  dA(z)\Big)^{\frac{1}{p}} <\infty, \ \  1\leq p<\infty\\
\sup_{z\in \CC}|f(z)|
e^{-\frac{|z|^2}{2}} <\infty, \ \ \ \  p= \infty,
\end{cases}
\end{align*} where $dA$ is the Lebesgue area measure on $\CC$. For each  function $f\in \mathcal{H}(\CC)$,  the subharmonicity of   $|f|^p$ implies that  the local point estimate
 \begin{align}
 \label{pointwise}
 |f(z)|^p
e^{-\frac{p|z|^2}{2}} \leq \int_{D(z,1)} |f(w)|^p
e^{-\frac{p|w|^2}{2}} dA(w)
 \end{align} holds  where  $D(z,1)$ is a disc of radius $1$ and center $z$.  This  implies further
  \begin{align}
 \label{global}
|f(z)|
\leq \bigg(\frac{2\pi}{p}\bigg)^{\frac{1}{p}} e^{\frac{|z|^2}{2}}  \|f\|_p.
 \end{align} By definition of the norm, the estimate  in \eqref{global}  is  valid for $p= \infty$ as well.

 \noindent
 The space $\mathcal{F}_2$ is a reproducing kernel Hilbert space with kernel function $K_w(z):= e^{\overline{w}z}$ and normalized kernel  $k_w:=\|K_w\|_2^{-1}K_w $. A straightforward calculation shows that  $k_w$ belongs to all  the Fock spaces  $\mathcal{F}_p$  with  a unit norm   $\|k_w\|_p= 1 $ for all   $w\in\CC$.

The rest of the manuscript is organized as follows. In Section~\ref{power}, we study the  power bounded property of the operators. As stated in Theorem~\ref{thm1} and Theorem~\ref{thm2},  these properties are  described in terms of  simple to apply conditions  which are merely based on  the values of the numbers  $|u(0)|$ or $|u(\frac{b}{1-a})|$ where  the constants $ a$ and $b$ are  from the  linear expansion of the symbol $\psi(z)= az+b$. The proofs of the results are presented in Subsections~\ref{pr1} and \ref{pr2}.
In Section~\ref{spectra},  we identify  the spectra of the operators  on $\mathcal{F}_p$  for all $1\leq p\leq \infty$; see Theorem~\ref{spectrum} whose proof is  given  in Subsection~\ref{prspec}. Section~\ref{uniform} contains several results  on the   uniform mean ergodic properties of the operators.

 We conclude this section with a word on notation.  The notion
 $U(z)\lesssim V(z)$ (or
equivalently $V(z)\gtrsim U(z)$) means that there is a constant
$C$ such that $U(z)\leq CV(z)$ holds for all $z$ in the set of a
question. We write $U(z)\simeq V(z)$ if both $U(z)\lesssim V(z)$
and $V(z)\lesssim U(z)$.
\section{Power bounded  $W_{(u,\psi)}$ }\label{power}
We start this section by  recalling certain definitions related to dynamics of an operator. Let $T$ be a bounded operator on a   Banach space $\mathcal{X}$. Then we set  the operator $T^n$ as the $n$-th iterate of  $T; T^n= T\circ T\circ ...\circ T$ $n$-times and $T^0= I$ where $I$ is the identity map on $\mathcal{X}$. The operator  $T$ is said to be power bounded  on  $\mathcal{X}$ if
    $\sup_{n\in \NN} \|T^n\| <\infty$. Obviously, any operator with norm  at  most $1$ is power bounded.   The notion of power boundedness or estimating  $ \|T^n\|$ plays an important roll in the study of  numerical stability of  initial value problems.    If
 \begin{align*}
 T_n:= \frac{1}{n}\sum_{m=1}^n T^m, \ \  n\in \NN
 \end{align*} denotes  its  the $n^{th}$ Ces$\grave{a}$ro means, then   $T$ is called
mean ergodic if there exists a bounded  operator $P$ on $\mathcal{X}$ such that  for each $f$ in $\mathcal{X}$
    \begin{align*}
    \lim_{n\to \infty}\|T_nf-Pf\|=0,
    \end{align*} and uniformly mean ergodic if the pointwise convergence above is uniform;
    \begin{align*}
    \lim_{n\to \infty}\|T_n-P\|=0.
    \end{align*}

     A straightforward simplification gives that  for each $n\in \NN$,  the relation
     \begin{align*}
     %\label{straight}
\frac{1}{n} T^n= T_n-\frac{n-1}{n} T_{n-1}
     \end{align*} holds  where we set $T_0= I $ as the identity operator on $\mathcal{X}$. This immediately  implies that
      if $T$ is mean ergodic, then
           $\frac{1}{n} T^n x\to 0$ as $n\to \infty$ for all $x\in \mathcal{X}.$ Similarly, $\frac{1}{n} \|T^n \| \to 0$ whenever $T$ is uniformly mean ergodic.  A number of authors have studied ergodicity of operators on various
     functional spaces; see for example \cite{ABJ,Beltran, JB3, JB1}. The   monographs   \cite{Kr,Y} provides
     basic information on ergodic theory.   Inspired by all these works,    the authors and J. Bonet in \cite{MMJ} studied the  mean ergodicity of composition operators acting on generalized Fock spaces, and  concluded  that  all  bounded composition operators  on  Fock spaces $\mathcal{F}_p$ are power bounded whenever $1\leq p\leq \infty$.  In this section,  we  show that  this conclusion   is no longer true in general for the  weighted composition operators $ W_{(u,\psi)}$. It is found that $ W_{(u,\psi)}$ is power bounded only when the weight function $u$ satisfies  an interesting point value  condition  as precisely stated  in the next two main theorems and proposition.

     The study of the dynamics of an operator  is related to the study of  its  iterates. For   $f\in \mathcal{H}(\CC)$, a simple argument shows that the image of $f$  under the iterates  of  $W_{(u,\psi)}$ has the form
\begin{align}
\label{interplay}
W_{(u,\psi)}^n f=  f(\psi^n) u_n, \ \ u_n(z):= \prod_{j=0}^{n-1} u(\psi^j(z))
\end{align}
for  each  $n\in \NN$ and $\psi^0= I$  the identity map on $\CC$. The equations in \eqref{interplay} will be repeatedly used  in the sequel.

We now state the main results on power boundedness.   Depending on whether  $|a|=1$ or $|a|<1$, we give two  main  results  in Theorem~\ref{thm1} and Theorem~\ref{thm2}.
        \begin{theorem}\label{thm1}
         Let $1\leq p\leq \infty , \ u,  \psi \in \mathcal{H}(\CC)$  and $W_{(u,\psi)}$ be bounded on  $\mathcal{F}_p$, and hence $\psi(z)= az+b,\ \ |a|\leq 1$.  If $|a|=1$, then the following statements are equivalent.
         \begin{enumerate}
         \item $W_{(u,\psi)}$ is power  bounded on  $\mathcal{F}_p$;
         \item $(\|u_n\|_p)_{n}$ is a bounded sequence;
         \item $|u(0)| \leq e^{-\frac{|b|^2}{2}}$. In this case for each $n\in\NN$
         \begin{align}
\label{ittt}
\| W_{(u,\psi)}^n \|= \Big(|u(0)| e^{\frac{|b|^2}{2} }\Big)^n.
\end{align}
         \end{enumerate}
          \end{theorem}
        It is interesting that we have  an easy to apply  equivalent conditions  for the
power boundedness of the weighted composition operators. Part (iii) of the condition is also    independent of  underlying space or the exponents $p$. In addition, by setting $n= 1$, the  theorem    provides  a simple expression for the norm of the  operators namely that
              \begin{align*}
              \|W_{(u,\psi)}\|=|u(0)| e^{\frac{|b|^2}{2}}.
              \end{align*}
We recall that a bounded linear operator is a contraction when its norm is bounded by $1$.  In view of this, we may add one more equivalent  condition to the above list in the theorem, namely that  $W_{(u,\psi)}$ is power  bounded on  $\mathcal{F}_p$ if and only if it is a contraction.    It should be also  noted that  for the case $ a\neq 1$ and  $|a|=1$, the above conditions are also  equivalent to  $\big|u\big(\frac{b}{1-a}\big)\big |\leq 1$ since an application of  Lemma~\ref{lem0} below  implies
\begin{align*}
\Big|u\Big(\frac{b}{1-a}\Big)\Big|=  |u(0)|\Big|K_{-\overline{a}b} \Big(\frac{b}{1-a}\Big)\Big|=  |u(0)| \Big|e^{-\frac{a |b|^2}{1-a}}\Big| \quad \quad \quad \quad \quad \\
=  |u(0)|e^{- |b|^2 \Re(\frac{a}{1-a})}=  |u(0)| e^{  \frac{|b|^2}{2}},
\end{align*} where $\Re$ denotes the real part of the given  complex number.   This inspires us  to ask whether a similar condition works for the remaining case namely that  when $|a|<1$. In this  case, as will be explained later, the powers of weighted composition operators are again weighted composition operators. This together with the relations in \eqref{norm}  and \eqref{norm2} ensure that the following necessary and sufficient conditions hold whenever $|a|<1$.
\begin{proposition}\label{prop}
Let $1\leq p\leq \infty, \ u,  \psi \in \mathcal{H}(\CC)$  and $W_{(u,\psi)}$ be bounded on  $\mathcal{F}_p$. Let  $\psi(z)= az+b$ and $|a|< 1$.
\begin{enumerate}
\item If $W_{(u,\psi)}$ is power  bounded on  $\mathcal{F}_p$, then
 \begin{align}
 \label{necc}
 \Big|u\Big(\frac{b}{1-a}\Big)\Big |\leq 1
 \end{align}
\item $W_{(u,\psi)}$ is power  bounded on  $\mathcal{F}_p$,  $p<\infty$   if
\begin{align}
\label{suff}
 \Big|u\Big(\frac{b}{1-a}\Big)\Big |\leq |a|^{\frac{2}{p}}
 \end{align}
\end{enumerate}
\end{proposition}
When $p\to \infty$, the right-hand side  in \eqref{suff} tends to  1. Thus, the condition in  \eqref{necc} is both necessary and sufficient for  $W_{(u,\psi)}$ to be power bounded  on the space  $\mathcal{F}_\infty$.   In particular when $W_{(u,\psi)}$ is   compact, we record our next main result which holds true on all the spaces  $\mathcal{F}_p$.

          \begin{theorem}\label{thm2}
          Let $1\leq p\leq \infty, \ u,  \psi \in \mathcal{H}(\CC)$  and $W_{(u,\psi)}$ be   bounded  on  $\mathcal{F}_p$, and   $\psi(z)= az+b, \ \ $ with $ |a|< 1 $.   Let $u$  be  non-vanishing and $W_{(u,\psi)}$ be  compact. Then the following  statements are equivalent.
              \begin{enumerate}
            \item $W_{(u,\psi)}$ is power  bounded on  $\mathcal{F}_p$;
         \item $(\|u_n\|_p)_{n}$ is a bounded sequence;
         \item $|u\big(\frac{b}{1-a}\big)|\leq 1 $.
             \end{enumerate}

 \end{theorem}
As in Theorem~\ref{thm1}, part (iii) of the condition  is  simple to apply and independent of the exponents $p$.  Observe that from the two theorems above, it is easy to see that a bounded composition
  operator $C_\psi$ is always power bounded while the multiplication  operator $M_u$ is not in general; see Corollary~\ref{corM}.

    To prove the results, we need to make some preparations.
The bounded and compact weighted composition operators on Fock spaces were characterized first in terms of Berezin-type integral transforms in  \cite{TM4,TM0,UK}. Later,  Le \cite{Tle} considered the Hilbert space $\mathcal{F}_2$ setting and obtained a simpler condition namely  that  $W_{(u,\psi)}$ is bounded on $\mathcal{F}_2$ if and only if
\begin{align}
 \label{bounded}
 M(u, \psi):=\sup_{z\in \CC} |u(z)|e^{\frac{1}{2}(|\psi(z)|^2-|z|^2)} <\infty.
 \end{align} He further proved that
 \eqref{bounded} implies $
 \psi(z)= az+b$ with $ |a|\leq 1$.  In \cite{TMMW}, T. Mengestie and M. Worku
 proved that the Berezin-type integral condition used to describe the boundedness of generalized Volterra-type integral operators $V_{(g,\psi)}$ on the Fock spaces  $\mathcal{F}_p$     is equivalent to  a simple condition as in  \eqref{bounded}. Because of the Littlewood-Paley type description of the Fock  spaces, by  simply  replacing $|g'(z)|/(1+|z|)$ by $|u(z)|$ in the results there, it has been known   that \eqref{bounded}  in fact describes
 the bounded weighted composition operators on  all the spaces  $\mathcal{F}_p, \ 1\leq p< \infty, $  with norm bounds
\begin{align}
 \label{norm}
  M(u, \psi)\leq \|W_{(u,\psi)}\|\leq
  |a|^{-\frac{2}{p}}  M(u, \psi).
 \end{align} For $p= \infty$, the corresponding  relation  holds with equality,
 \begin{align}
 \label{norm2}
 \|W_{(u,\psi)}\| =  M(u, \psi).
 \end{align}
 % The same conclusion as in \eqref{norm} was also reported later  in \cite{TTK0} for $p<\infty$.
   As indicated in \cite{Tle}, an  interesting consequence of \eqref{bounded} is that  if $|a|=1$, then a simple argument with Liouville's theorem gives  that the weight function  $u $  has the form $ u(z)= u(0)K_{-\overline{a}b}(z). $  This representation of  $u$ will play an important roll in the rest of the paper. Thus, we may  formulate it as  a lemma for the purpose of  easy further   referencing.
 \begin{lemma}
 \label{lem0}
 Let $1\leq p\leq \infty, \ u, \  \psi \in \mathcal{H}(\CC)$ and $W_{(u,\psi)}$ be bounded on $\mathcal{F}_p,$ and hence $\psi(z)= az+b, \  |a|\leq1$.  If $|a|=1$, then
 \begin{align*}
  u(z)= u(0)K_{-\overline{a}b}(z).
 \end{align*}
 \end{lemma}
 It should be noted that  by condition \eqref{bounded} it is possible for $W_{(u,\psi)}= M_u C_\psi$ to be bounded   even if both the factors  $C_\psi$ and $M_u$ are  unbounded. The functions  $u(z)= z$ and $\psi(z)= z+1$ provides such an example.

 Similarly, compactness of $W_{(u,\psi)}$ has been described by the fact that $\psi(z)= az+b, |a|\leq 1$ and  $|u(z)|e^{\frac{1}{2}(|\psi(z)|^2-|z|^2)}\to 0$ as $|z|\to \infty$.  The later condition  implies that $|a|<1$ but not conversely. Very recently, Carroll and Gilmore \cite{CG}, used the idea of order of analytic function  and proved the following analogues result.
 \begin{lemma}\label{newlem}
  Let $1\leq p\leq \infty, \ u, \  \psi \in \mathcal{H}(\CC)$ and  $\psi(z)= az+b, \  |a|<1$, and assume that $u$ is non-vanishing. Then  $W_{(u,\psi)}$ is
   compact on $\mathcal{F}_p$ if and only if $ u$ has the form
  \begin{align}
  \label{neww}
  u(z)= e^{a_0+a_1z+ a_2z^2}
    \end{align} for some constants $a_0, a_1, a_2$ such that $|a_2|<\frac{1-|a|^2}{2}$.
 \end{lemma}

Next, we  consider the  following key  necessary  conditions for power bounded  $ W_{(u,\psi)}$. The lemma gives  a good restriction on the growth of the sequence $(\|u_n\|_p)_n $ and  the value  $|u(z_0)|$ where $z_0$ is a fixed point of $\psi$.
        \begin{lemma}
        \label{lem2}
 Let $1\leq p\leq \infty$ and $ \ u, \psi \in \mathcal{H}(\CC)$. If  $W_{(u,\psi)}$ is
  power  bounded on  $\mathcal{F}_p$, then
 \begin{enumerate}
\item $|u(z_0)|\leq 1$  where $z_0$ is a fixed point of $\psi$:
\item $(\|u_n\|_p)_{n}$ is a bounded sequence.
\end{enumerate}
\end{lemma}
\begin{proof}
(i). Since the constant function \textbf{1} belongs to the spaces $\mathcal{F}_p$ with $\|\textbf{1}\|_p=1 $, using  the pointwise estimate in \eqref{global}
\begin{align*}
\|W^n_{(u,\psi)}\|\geq \|W^n_{(u,\psi)} \textbf{1}\|_p\gtrsim  |W^n_{(u,\psi)}\textbf{ 1}(z_0)|e^{-\frac{|z_0|^2}{2}}=  |u_n(z_0)|e^{-\frac{|z_0|^2}{2}}
=  |u(z_0)|^n e^{-\frac{|z_0|^2}{2}}
\end{align*} from which  the inequalities
\begin{align*}
\infty >\sup_{n\in\NN}\|W^n_{(u,\psi)}\|\geq  e^{-\frac{|z_0|^2}{2}} \sup_{n\in\NN} |u(z_0)|^n
\end{align*} hold  only if $|u(z_0)|\leq 1.$

 To prove (ii),  for $p= \infty $ arguing as above we have
\begin{align*}
\|W^n_{(u,\psi)}\|\geq \|W^n_{(u,\psi)} \textbf{1}\|_\infty \geq  |W^n_{(u,\psi)}\textbf{ 1}(z)|e^{-\frac{|z|^2}{2}}= |u_n(z)|e^{-\frac{|z|^2}{2}}.
\end{align*}
Taking the  supremum  with respect to  first with  $z$ and then with $n$ give the required assertion.
On the other hand, if $p<\infty$, then
\begin{align*}
\|W^n_{(u,\psi)}\|^p \geq \|W^n_{(u,\psi)} \textbf{1}\|_p^p=\frac{p}{2\pi}\int_{\CC} |u_n(z)|^pe^{-\frac{p|z|^2}{2}} dA(z)= \|u_n\|_p^p
\end{align*}  from which the conclusion follows again.
\end{proof}

The next  simple  lemma will be crucial  in the proof of Theorem~\ref{thm2}.
\begin{lemma}\label{techlemma}
 Let $a\in \CC$ and $|a|<1$. Then  for all $n\in \NN$
\begin{align}
\label{Induction}
\frac{|1-a^2|}{1-|a|^2}\geq \frac{|1-a^{2n}|}{1-|a|^{2n}}.
\end{align}
\end{lemma}
 \begin{proof}
Applying triangular inequality,
\begin{align*}
 \frac{|1-a^{2n}|}{|1-a^{2}|} =    \left|1 + a^{2} + (a^{2})^{2} + ... +(a^{2})^{n-1}\right|  \leq 1+|a^{2}|+ |(a^{2})^2|+...+ |(a^{2})^{(n-1)}|  \\
 = 1 + |a|^{2} + (|a|^{2})^{2} + ... + (|a|^{2})^{n-1} = \frac{1-|a|^{2n}}{1-|a|^2} \quad \quad \quad \quad \quad \quad \quad \quad \quad \quad
\end{align*} from which \eqref{Induction} follows.
\end{proof}
We are now ready to give the  proofs of the previous two main results.
\subsection{Proof of Theorem~\ref{thm1}} \label{pr1}
 The statement (i) implies (ii) is proved in Lemma~\ref{lem2}. On the other hand, if (ii) holds, then  using \eqref{global}

               \begin{align}
               \label{more}
               \infty > \sup_{n\in\NN_0}\|u_n\|_p \gtrsim \sup_{n\in\NN_0} |u_n(z)|e^{-\frac{|z|^2}{2}}
               \end{align} for each $z\in \CC$.
                If $a= 1$, then  $\psi^j(z)= z+jb$ and using Lemma~\ref{lem0},
\begin{align*}
u_n(z)= u(0)^n\prod_{j=0}^{n-1}K_{-b}(z+jb)=  u(0)^n e^{l_n(z)}
\end{align*}
where
\begin{align*}
l_n(z):=-\overline{b} \sum_{j=0}^{n-1}(z+jb)= -\overline{b}nz-\frac{|b|^2}{2} n(n-1).
\end{align*} It follows that
\begin{align}
\label{111}  u_n(z)=  u(0)^n e^{-\frac{|b|^2}{2} n(n-1)} K_{-nb}(z)
\end{align}
for all $z\in \CC$. Considering  \eqref{111} and  applying the  estimate in \eqref{more} at $z= -nb$
\begin{align*}
\sup_{n\in\NN_0}\|u_n\|_p \gtrsim  \sup_{n\in\NN_0} |u_n(-nb)|e^{-\frac{|nb|^2}{2}}\quad \quad \quad \quad\quad \quad \quad \quad\quad \quad\nonumber \\
 = \sup_{n\in\NN_0}\bigg|u(0)e^{\frac{|b|^2}{2}}\bigg|^{n}  e^{-|b|^2 n^2} K_{-nb}(-nb)= \sup_{n\in\NN_0}\bigg|u(0)e^{\frac{|b|^2}{2}}\bigg|^{n}
\end{align*}
 and hence the statement in (iii) follows. On the other hand, if $a\neq1$ and $|a|=1$, then set $z_0$ be  the fixed point of $\psi$ and eventually applying Lemma~\ref{lem0}
               \begin{align*}
               | u_n(z_0)|=\left|u(0)K_{-\overline{a}b}\left( \frac{b}{1-a} \right) \right|^n= \left|u(0)e^{-a\overline{b}\left( \frac{b}{1-a} \right)} \right|^n \quad \quad \quad \quad \quad \quad \\
=\left|u(0) \right|^ne^{-n\Re  \left( \frac{a|b|^{2}}{1-a} \right)}= |u(0)|^ne^{\frac{n|b|^2}{2} }
               \end{align*} and the conclusion follows after taking this in \eqref{more} again.

 It remains to prove (iii) implies (i).  First observe that for each $n\in \NN$,   the operator $W_{(u,\psi)}^n$ itself is a weighted composition operator   and
$ W_{(u,\psi)}^n= M_{u_n}C_{\psi^n}= W_{(u_n,\psi^n)} $. Then   $W_{(u,\psi)}$ is power bounded if and only if
  \begin{align*}
             \sup_{n\in \NN} \sup_{z \in \CC} |u_n(z)| e^{\frac{1}{2}\big( \left| a^n z + \frac{b(1-a^{n})}{1-a} \right|^2-|z|^2\big)}<\infty.
              \end{align*}
  Thus, for $|a|=1$, we apply \eqref{norm} and obtain the  norm
\begin{align}
\label{itself}
\|W_{(u,\psi)}^n\|= \sup_{z\in \CC} |u_n(z)| e^{\frac{1}{2}(|\psi^n(z)|^2-|z|^2)}.
\end{align}
Our next task is to simplify \eqref{itself}.  If $a= 1$, then the representation in   \eqref{111} implies
 \begin{align}
\label{it}
\| W_{(u,\psi)}^n \|= \sup_{z\in \CC}e^{\frac{1}{2}(|z+nb|^2-|z|^2)}| u(0)|^n e^{-\frac{|b|^2}{2} n(n-1)}| K_{-nb}(z)|\quad \quad \quad \ \nonumber\\
= \Big(|u(0)| e^{\frac{|b|^2}{2} }\Big)^n \sup_{z\in \CC}e^{\Re(nb\overline{z})}| K_{-nb}(z)|=\Big(|u(0)| e^{\frac{|b|^2}{2} }\Big)^n
\end{align} from which the statement  follows.

\noindent   Next, assume  $a\neq1$ and $|a|=1$. Then  $\psi^j(z)= a^jz+b\frac{1-a^j}{1-a}$. By using  Lemma~\ref{lem0} again   $u_n(z)= u(0)^ne^{h_n(z)} $
 where
 \begin{align*}
 %\label{alww}
 h_n(z):=-a\overline{b} \sum_{j=0}^{n-1}\Big(a^jz+b\frac{1-a^j}{1-a}\Big)= -a\overline{b}z\frac{1-a^n}{1-a} -\frac{a|b|^2n}{1-a} + \frac{a|b|^2(1-a^n)}{(1-a)^2}.
 \end{align*}
 Thus, we have
 \begin{align*}
u_n(z)= u(0)^ne^{ -\frac{a|b|^2n}{1-a} + \frac{a|b|^2(1-a^n)}{(1-a)^2}} K_{-\overline{a}b\frac{1-\overline{a}^n}{1-\overline{a}}}(z)
 \end{align*}
 from which and \eqref{itself}
  \begin{align*}
\| W_{(u,\psi)}^n \|= \sup_{z\in \CC}e^{\frac{1}{2} \big|a^nz+b\frac{1-a^n}{1-a}\big|^2-\frac{1}{2}|z|^2} |u(0)|^n\bigg|e^{-a\overline{b}z\frac{1-a^n}{1-a}-\frac{a|b|^2n}{1-a}  +\frac{a|b|^2(1-a^n)}{(1-a)^2} }\bigg| \quad \quad \quad \quad
\end{align*}
where
\begin{align*}
 \Big|a^nz+b\frac{1-a^n}{1-a}\Big|^2-|z|^2=  |b|^2\bigg|\frac{1-a^n}{1-a}\bigg|^2 + 2\Re\Big( a^nz \overline{b}\frac{1-\overline{a}^n}{1-\overline{a}}\Big)\quad \quad
 \end{align*}
and
\begin{align*}
\bigg|e^{-a\overline{b}z\frac{1-a^n}{1-a} -\frac{a|b|^2n}{1-a} +\frac{a|b|^2(1-a^n)}{(1-a)^2} }\bigg|=e^{\Re\big(-a\overline{b}z\frac{1-a^n}{1-a}-\frac{a|b|^2n}{1-a}  +\frac{a|b|^2(1-a^n)}{(1-a)^2}\big) }.
\end{align*}
On the other hand,
\begin{align*}
a^nz\overline{b}\frac{1-\overline{a}^n}{1-\overline{a}}-a\overline{b}z\frac{1-a^n}{1-a}= z\overline{b}(a^n-1)\Big(\frac{1}{1-\overline{a}}+ \frac{a}{1-a}\Big)= 0
\end{align*} and  combining all the above
\begin{align}
\label{alwww}
\| W_{(u,\psi)}^n \|= |u(0)|^n e^{\frac{|b|^2}{2}\big|\frac{1-a^n}{1-a}\big|^2+\
 \Re\big(\frac{a|b|^2(1-a^n)}{(1-a)^2}-\frac{a|b|^2n}{1-a}\big)} \quad \quad \quad \quad  \quad \quad \quad  \quad \quad \nonumber\\
 \leq |u(0)|^n e^{\frac{|b|^2}{2}\big|\frac{2}{1-a}\big|^2+
 \Re \left(\frac{a|b|^2(1-a^n)}{(1-a)^2}\right) -\Re \left( \frac{a|b|^2n}{1-a} \right)} \quad \quad \quad \nonumber \\
 \leq |u(0)|^n e^{\frac{|b|^2}{2}\big|\frac{2}{1-a}\big|^2+
 \left|\frac{a|b|^2(1-a^n)}{(1-a)^2}\right|} e^{ -n |b|^2\Re \left( \frac{a}{1-a} \right)}\nonumber\\
 \leq  e^{\frac{|b|^2}{2}\big|\frac{2}{1-a}\big|^2+
\frac{2|b|^2}{|1-a|^2}} \Big( |u(0)|e^{ -|b|^2\Re \left( \frac{a}{1-a} \right)} \Big)^{n}\nonumber\\
= e^{\frac{|b|^2}{2}\big|\frac{2}{1-a}\big|^2+
\frac{2|b|^2}{|1-a|^2}}\Big( |u(0)|e^ {\frac{|b|^2}{2} } \Big)^{n}.
\end{align}
Thus, power boundedness follows whenever $ |u(0)|\leq e^{ -\frac{|b|^2}{2} }$.
\subsection{Proof of Theorem~\ref{thm2}} \label{pr2}
  The statement (i) implies (ii) follows from Lemma~\ref{lem2} again. Assuming (ii), we proceed to show that (iii) holds. Using \eqref{global} we  estimate
\begin{align}
\label{inff}
\infty > \sup_{n\in \NN_0} \|u_n\|_p \gtrsim \sup_{n\in \NN_0} |u_n(z_0)|e^{-\frac{|z_0|^2}{2}}.
\end{align} where $z_0= b/(1-a)$ is the fixed point of $\psi$.  Moreover, observe that
\begin{align*}|u_n(z_0)|= \prod_{j=0}^{n-1} u\big(\psi^j(z_0)\big)=   |u(z_0)|^n
\end{align*} which together with Lemma~\ref{lem0} and  \eqref{inff} gives statement (iii).

\noindent
Next, we prove (iii) implies (i). If  $ p = \infty $,     applying the relation  in \eqref{norm2} for the  weighted composition operator   $ W_{(u_n,\psi^n)} $  we get
\begin{align*}
\|W_{(u_n,\psi^n)} \|= \|W_{(u,\psi)}^n \|= \sup_{z\in \CC} |u_n(z)| e^{\frac{1}{2}\big( \left| a^n z +  z_0 b(1-a^{n}) \right|^2-|z|^2\big)}.
\end{align*} Thus, $ W_{(u_n,\psi^n)} $ is power bounded on $\mathcal{F}_\infty$  if and only if
 \begin{align}
\label{power2}
\sup_{n\in \NN}\sup_{z \in \CC} |u_n(z)| e^{\frac{1}{2}\big( \left| a^n z + z_0 b(1-a^{n})\right|^2-|z|^2\big)}<\infty.
\end{align}Therefore,
by using the assumption   $\big|u \left( \frac{b}{1-a}\right) \Big|\leq 1$, we plan to show that \eqref{power2} holds.
First we consider  Lemma~\ref{newlem} and compute
\begin{align*}
u_n(z)= \prod_{j=0}^{n-1} u(\psi^j(z))=e^{S_n(z)}
\end{align*} where
\begin{align}
\label{exp}
S_n(z)= \sum_{j=0}^{n-1} \bigg(a_0+ a_1 \Big( a^jz+ \frac{(1-a^j)b}{1-a}\Big)+ a_2\Big( a^jz+\frac{(1-a^j)b}{1-a}\Big)^2\bigg)\quad \quad \quad \nonumber\\
= na_0 + \frac{a_1bn}{1-a}- \frac{a_1b(1-a^n)}{(1-a)^2}+  \frac{a_1(1-a^n)}{1-a}z+   \frac{a_2(1-a^{2n})}{1-a^2} z^2\nonumber\\
+\frac{a_2b^2}{(1-a)^2}\bigg(n-\frac{2(1-a^n)}{1-a}+\frac{1-a^{2n}}{1-a^2}\bigg) +\frac{2a_2zb}{1-a}\bigg(\frac{1-a^n}{1-a}-\frac{1-a^{2n}}{1-a^2} \bigg)\nonumber\\
= na_0+  \frac{a_1bn}{1-a}- \frac{a_1b(1-a^n)}{(1-a)^2}+ \frac{a_2b^2}{(1-a)^2}\bigg(n-\frac{2(1-a^n)}{1-a}+\frac{1-a^{2n}}{1-a^2}\bigg)\nonumber\\
+   \bigg(\frac{a_1(1-a^n)}{1-a}+   \frac{2a_2b}{1-a}\bigg(\frac{1-a^n}{1-a}-\frac{1-a^{2n}}{1-a^2} \bigg)\bigg) z+ \frac{a_2(1-a^{2n})}{1-a^2} z^2.
\end{align}
Now taking this into account and the fact that $ W_{(u,\psi)}^n=  W_{(u_n,\psi^n)} $, the corresponding notation in \eqref{bounded} becomes
\begin{align}
\label{newsup}
M(u_n, \psi^n)= \sup_{z\in \CC} |u_n(z)|e^{\frac{1}{2}\big( \left| a^n z + \frac{b(1-a^{n})}{1-a} \right|^2-|z|^2\big)}\nonumber\\
= e^{c_n} \sup_{z\in \CC} e^{\Re(t_nz)+ \Re(p_nz^2)-q_n|z|^2}\nonumber\\
\leq e^{c_n} \sup_{z\in \CC} e^{\Re(t_nz)+ (|p_n|-q_n)|z|^2}
\end{align} where $c_n$ is the real part of the expression
$$
 na_0 + \frac{a_1bn}{1-a}- \frac{a_1b(1-a^n)}{1-a}^2+ \frac{a_2b^2}{(1-a)^2}\bigg(n-\frac{2(1-a^n)}{1-a}+\frac{1-a^{2n}}{1-a^2}\bigg)+ \Big|\frac{b(1-a^{n})}{1-a} \Big|^{2},
$$

$$
t_n= \frac{a_1(1-a^n)}{1-a}+   \frac{2a_2b}{1-a}\bigg(\frac{1-a^n}{1-a}-\frac{1-a^{2n}}{1-a^2} \bigg)+ \overline{\frac{(1-a^n)b}{1-\overline{a}}} a^n,
$$
$$
p_n=  \frac{a_2(1-a^{2n})}{1-a^2}, \ \ \text{and} \ \ \  q_n= \frac{1-|a|^{2n}}{2}.
$$
Now to estimate the supremum in \eqref{newsup}, we claim that
$$|p_n|-q_n= |a_2|\Big| \frac{1-a^{2n}}{1-a^2} \Big| - \frac{1-|a|^{2n}}{2}<0.$$
Observe   that the  inequality holds  if and only if
\begin{align}
|a_2| <\frac{(1-|a|^{2n})}{\big|1-a^{2n}\big|}  \frac{\big|1-a^2\big|}{2}.
\end{align}  This follows immediately from Lemma~\ref{techlemma}  as  $|a_2|<\frac{1-|a|^2}{2}$.

It follows from this and  \eqref{newsup} that
\begin{align*}
M(u_n, \psi^n) \lesssim  e^{c_n}.
\end{align*}
On the other hand, since $|a|<1$
\begin{align*}
c_n \leq  n\Re\bigg(a_0 + \frac{a_1b}{1-a}+\frac{a_2b^2}{(1-a)^2}\bigg) + C
\end{align*}  for some positive constant $C$ and hence
$$
e^{c_n} \lesssim e^{n\Re\big(a_0 + \frac{a_1b}{1-a}+\frac{a_2b^2}{(1-a)^2}\big)}\simeq  \Big|u\Big (\frac{b}{1-a}\Big)\Big|^n
$$
 from which  and the assumption that $\Big|u\big(\frac{b}{1-a}\big)\Big|\leq 1$,  the condition in  \eqref{power2} follows.

Next, we consider the case  when $p <\infty$ and consider first the case $a=0$.  Then $\psi^n(z)=b$,  $ u_n(z)=  u(z)(u(b))^{n-1}$ and  applying \eqref{global},
\begin{align*}
\| W_{(u,\psi)}^{n} f \|^{p}_{p} = \frac{p}{2\pi}\int_{\CC} \big| f ( b ) \big|^{p} |u_n(z)|^p e^{-\frac{p}{2}|z|^2}dA(z) \quad \quad \quad \quad \quad \quad \quad\\
= \frac{p}{2\pi}\int_{\CC} \big| f ( b ) \big|^{p} |u(z)|^p |u(b)|^{(n-1)p} e^{-\frac{p}{2}|z|^2}dA(z) \quad \quad \quad \quad \quad \quad\\
=|f(b)|^{p}|(u(b))|^{(n-1)p} \| u \|_{p}^{p}\leq |(u(b))|^{(n-1)p} \| u \|_{p}^{p} e^{|b|^{2}}  \| f \|^{p}_{p}.
\end{align*}From which we arrive at  the claim. Here note that since $W_{(u,\psi)}$ is bounded, the multiplier $u$ belongs to $\mathcal{F}_p$ for all $p$.

 If $a\neq0$,  then applying    the local point estimate in \eqref{pointwise},
\begin{align}
\label{Fubini}
\| W_{(u,\psi)}^{n} f \|^{p}_{p} = \frac{p}{2\pi}\int_{\CC} \left| f \left( a^n z + \frac{b(1-a^{n})}{1-a} \right) \right|^{p} |u_n(z)|^p e^{-\frac{p}{2}|z|^2}dA(z) \nonumber \\
\leq \frac{p}{2\pi}\int_{\CC} e^{\frac{p}{2} \big(\left| a^n z + \frac{b(1-a^{n})}{1-a} \right|^2-|z|^2\big)} |u_n(z)|^p\quad \quad \nonumber\\
\times \int_{D\big(a^n z + \frac{b(1-a^{n})}{1-a}, 1\big)} | f (w)|^{p}e^{-\frac{p}{2}|w|^2} dA(w)dA(z)\nonumber \\
=\frac{p}{2\pi}\int_{\CC}\int_{\CC} e^{\frac{p}{2} \big(\left| a^n z + \frac{b(1-a^{n})}{1-a} \right|^2-|z|^2\big)} |u_n(z)|^p\quad \quad \quad \nonumber\\
\chi_{D\big(a^n z + \frac{b(1-a^{n})}{1-a}, 1\big)}(w) | f (w)|^{p}e^{-\frac{p}{2}|w|^2} dA(w)dA(z).\quad \quad \quad
\end{align}
Observe that if  $w\in D\big(a^n z + \frac{b(1-a^{n})}{1-a}, 1\big)$,  then
\begin{align*}
1\geq \bigg| w- a^n z - \frac{b(1-a^{n})}{1-a}\bigg|= |a|^n \bigg| \frac{w}{a^n}-  z - \frac{b(1-a^{n})}{a^n(1-a)}\bigg|\\
= |a|^n \bigg| z-\Big(\frac{w}{a^n} - \frac{b(1-a^{n})}{a^n(1-a)}\Big)\bigg|,
\end{align*}    which holds true if and only if
\begin{align*}
\frac{1}{|a|^n} \geq \bigg| z-\Big(\frac{w}{a^n} - \frac{b(1-a^{n})}{a^n(1-a)}\Big)\bigg|.
\end{align*} Thus, $w$ belongs to the disk
\begin{align*}
 D\bigg(a^n z + \frac{b(1-a^{n})}{1-a}, 1\bigg)
 \end{align*}
  if and only if  $z$ belongs to
\begin{align*}D\bigg(\frac{w}{a^n} -\frac{b(1-a^{n})}{1-a}, \frac{1}{|a|^n}\bigg).\end{align*}
Making use of this and Fubini's theorem in \eqref{Fubini}
\begin{align*}
%\label{iff}
\| W_{(u,\psi)}^{n} f \|^{p}_{p}
\leq \frac{p}{2\pi} \int_{\CC}  | f (w)|^{p}e^{-\frac{p}{2}|w|^2} \quad\quad \nonumber\\
\times  \bigg(\int_{\CC} e^{\frac{p}{2} \big(\Big| a^n z + \frac{b(1-a^{n})}{1-a} \Big|^2-|z|^2\big)} |u_n(z)|^p
\chi_{D\big(\frac{w}{a^n} -\frac{b(1-a^{n})}{1-a}, \frac{1}{|a|^n}\big)}(z)  dA(z)\bigg)dA(w).
\end{align*}
Using Lemma~\ref{newlem} and simplifying  like  the case for   $p= \infty$, we get
\begin{align*}
e^{\frac{p}{2} \big(\Big| a^n z + \frac{b(1-a^{n})}{1-a} \Big|^2-|z|^2\big)} |u_n(z)|^p
\leq e^{pc_n} e^{p|t_n||z|-p\big(\frac{1-|a|^{2n}}{2}-\big| \frac{a_2(1-a^{2n})}{1-a^2}\big|\big)|z|^2}
\end{align*} for all $z\in \CC$. Since $ W_{(u,\psi)} $  is compact  and  $|a_2|<\frac{1-|a|^2}{2}$, it follows that
$$ q_n-|p_n|=\frac{1-|a|^{2n}}{2}-\Big| \frac{a_2(1-a^{2n})}{1-a^2}\Big|>0$$
 and
\begin{align*}
\int_{\CC} e^{\frac{p}{2} \big(\Big| a^n z + \frac{b(1-a^{n})}{1-a} \Big|^2-|z|^2\big)} |u_n(z)|^p
\chi_{D\big(\frac{w}{a^n} -\frac{b(1-a^{n})}{1-a}, \frac{1}{|a|^n}\big)}(z)  dA(z)\quad \quad \quad \quad \quad \\
\leq \int_{\CC} e^{pc_n} e^{p|t_n||z|-p\big(\frac{1-|a|^{2n}}{2}-\big| \frac{a_2(1-a^{2n})}{1-a^2}\big|\big)|z|^2} dA(z)
\lesssim e^{pc_n}
\end{align*} where we used the fact that $\sup_{n\in \NN} t_n <\infty$ and $ \inf_n\{ q_n-|p_n|\}>0$.
Hence,
\begin{align*}
%\label{ifff}
\| W_{(u,\psi)}^{n} f \|_p\lesssim e^{c_n}\|f\|_p
\end{align*}
and  the conclusion follows as in the last part of $p= \infty$, and   completes the proof of the theorem.

\section{The spectrum of $W_{(u,\psi)}$ }\label{spectra}
In this section,  we study the spectral  of  weighted compositions operators on  all the  Fock spaces $ \mathcal{F}_p, 1\leq p\leq \infty$.
Let $T$ be a bounded linear operator on a Banach space $\mathcal{X}$. Then the  spectrum $\sigma (T)$ of $T$  is the set
$\{\lambda\in \CC: T-\lambda I\ \ \text{is not invertible}\} $ where $I$ is the identity operator on $\mathcal{X}$. The spectrum $\sigma (T)$ is always a non-empty compact and closed subset of the disk centered at the origin and of radius $\|T\|$.   It has been well-known that
the spectrum of an operator plays a vital roll in the study of its  dynamical properties; see \cite{Grosse}. Our next result will be used to
prove  mean ergodic results in the next section  apart from being interest of its own.
\begin{theorem}\label{spectrum}
 Let $1\leq p\leq \infty, \ u,  \ \psi \in \mathcal{H}(\CC)$ and  $W_{(u,\psi)}$ be  bounded  on $\mathcal{F}_p$ and hence $\psi(z)= az+b$ with $|a|\leq1$. Then if
 \begin{enumerate}
 \item  $W_{(u,\psi)}$  is compact and hence $|a|<1$, then
 \begin{align}
 \label{spectrum}
\sigma(W_{(u,\psi)})=
\Big\{0, \  u\Big(\frac{b}{1-a}\Big) a^m, \ \  m\in \NN_0\Big\}.
\end{align}
\item $|a|=1$, then
 \begin{align*}
% \label{spectrum2}
\sigma(W_{(u,\psi)})=
\begin{cases}
\overline{\Big\{ u(0)e^{\frac{a|b|^2}{a-1}} a^m: m\in \NN_0 \Big\}},  \ a\neq1 \\
\left\{ z: |z|  = |u(0)| e^{\frac{|b|^{2}}{2}}  \right\}, \ \ \ a=1, \ b\neq 0\\
\{u(0)\}, \ \ \ a=1, \ \ b= 0.
\end{cases}
\end{align*}
\end{enumerate}
\end{theorem}
\noindent
We now remark a few points. First, observe that the number $b/(1-a)$ in \eqref{spectrum} is the fixed point of the symbol $\psi$.  We also note that when $a\neq1$ and $|a|=1$ the expression in the spectrum can be expressed in terms of this fixed point. That is from Lemma~\ref{lem0},  it follows that $ u(0)e^{\frac{a|b|^2}{a-1}} a^m=  u(b/(1-a))a^m$. In this case,   the spectrum contains finite number of points only when
$a$ is  a root of unity.
Next, for $a\neq1,$ the necessity of the condition that $|u(b/(1-a))|\leq 1$ in Theorem~\ref{thm1} and Theorem~\ref{thm2} can be easily deduced using Theorem~\ref{spectrum} since the spectrum of a power bounded operator is always a subset of the closed unit disc. Here  for $|a|=1$ and $a\neq1$,  note that
$\big| u(0)e^{\frac{a|b|^2}{a-1}}\big| =  \big|u\big(\frac{b}{1-a}\big)\big|$. Another application of Theorem~\ref{spectrum} will be given in Section~\ref{uniform}. It is known that the spectrum plays an essential roll in the study of theory of semigroups of linear operators \cite{EN} as well.
\subsection{Proof of Theorem~\ref{spectrum}}\label{prspec}
(i).
  Let $W_{(u,\psi)}$  is compact and hence $|a|<1$.
Here our proof is based on an argument  that goes back to \cite{HK}. We set $z_0= b/(1-a)$ and  plan to show that  the range of $W_{(u,\psi)}- a^nu(z_0)I$ fails to contain the complex polynomial $z^n$. Setting $n= 1$ and arguing in the direction of
contradiction, assume that there exists an $f\in \mathcal{F}_p$  such that
\begin{align}
\label{one}
u(z) f(\psi(z))- au(z_0) f(z)= z.
\end{align}
If $u(z_0)= 0$ or a= 0, then $au(z_0)= 0$ and belongs to the spectrum. Thus,  we may assume that $z_0$ is not in the zero set of $u$ and $a\neq0$.
 First assume that $ z_0=0$. Then  taking $z=0$ in \eqref{one},  we obtain that $f(0)=0$.
 On the other hand, differentiating both side of equation \eqref{one} and setting again $z=0$, we obtain
\begin{align*}
u'(0) f(\psi(0))+ u(0) \psi'(0) f'(\psi(0))-au(0)f'(0)= 1
\end{align*} which results  the contradiction $0=1$.

\noindent Similarly   for $n>1$,  differentiating  both side of the equation
 \begin{align*}u(z) f(\psi(z))- a^nu(z_0) f(z)= z^n
 \end{align*}repeatedly and eventually setting $z=0,$    we obtain
$
 f^{(m)}(0)=0
$ for all $m<n$ while for $m=n$ we get  again the contradiction  $0= n!$.\\

\noindent  If $ z_0\neq0,$ then  we may set $\psi_1(z)=  az$,
  \begin{align*}
  u_1(z)= \frac{u(z+z_0)}{\|K_{-z_0}\|_2^{2} } e^{-\overline{z_0}z+ \overline{z_0}(az+az_0+b)}
 \end{align*} and observe that $\psi_1(0)= 0$ and $u_1(0)= u(z_0)$. A  straightforward calculation  shows that
 \begin{align*}
 W_{(u_2,\psi_2)} W_{(u, \psi)} W_{(u_2,\psi_2)}^{-1}= W_{(u_1, \psi_1)}
 \end{align*}  and     $ W_{(u_2,\psi_2)}^{-1}= W_{(u_3, \psi_3)}$ where  $u_2(z)= k_{-z_0}(z)$,  $\psi_2(z)= z+z_0$, $u_3=k_{z_0}$,   and  $\psi_3(z)=z-z_0$.
It follows that  the weighted composition operators  $W_{(u_1, \psi_1)}$ and $W_{(u, \psi)}$ are similar  and
   have the same spectrum, and our conclusion follows from the first  case discussed above.     Therefore, the set in the right-hand side of \eqref{spectrum} in this case  is contained in the spectrum.\\
Conversely, if  $|a|<1$, then $W_{(u, \psi)} $ is compact and its spectrum contains only zero and eigenvalues. Thus,
we consider a nonzero eigenvalue $\lambda\in  \sigma(W_{(u,\psi)}$  and  show that it is of  the form  $u(z_0)a^n$ for some positive integer $n$.
 If $f$ is a  corresponding nonzero eigenvector, then
  \begin{align}
  \label{eigen}
 W_{( u,\psi)} f(z)=  u(z)f(az+b)= \lambda f(z)
  \end{align}
   for all $z$ in $\CC$.  If $f$ has no zero at $z_0$, then   \eqref{eigen} implies
  $  u(z_0)= \lambda
 $ and hence $\lambda=  a^0 u(z_0)$. On the other hand, if $f$ has  zero at $z_0$ of order $m$, we may write
 \begin{align*}
 f(z)= (z-z_0)^m g(z)
 \end{align*} where $g(z_0)\neq 0$.  Then substituting $f$ by   this  in \eqref{eigen} and
  differentiating both sides of the equation  $m$ times and eventually
  setting $z=z_0$, we only  get
  \begin{align}
  \label{last}
a^m u(z_0) m! g(z_0)= \lambda m! g(z_0)
  \end{align} as all the other terms have factor $z-z_0$ and vanish. Now, $g(z_0)$ is non-zero and  \eqref{last}  holds only if  $\lambda=a^m u(z_0)$ as asserted.

 The argument in the proof of  part (ii) is divided into  three  cases depending on the values of $a$ and $b$. \\
\noindent\emph{Case 1}. Let  $|a|=1$ and  $a\neq1$. For simplicity we first set $\psi_{z_0}(z) = z-z_0$ and claim that  the weighted composition operator induced by $(k_{z_0}, \psi_{z_0})$ is an isometric bijective map on $\mathcal{F}_p$ with inverse  $W_{(k_{-z_0},\psi_{z_0}^{-1})}$. To this claim, for every  $f\in \mathcal{F}_{p}$
\begin{align*}
   \|W_{(k_{z_0},\psi_{z_0})}f\|_p^p   = \frac{p}{2\pi} \int_{\CC} |k_{z_0}(z)|^p |f(z-z_0)|^p e^{-\frac{p}{2}|z|^2}dA(z) \quad \quad \quad \quad \\
   =\frac{p}{2\pi}  \int_{\CC}  |f(z-z_0)|^p e^{-\frac{p}{2}|z-z_0|^2} \bigg( |k_{z_0}(z)|^p e^{\frac{p}{2}|z-z_0|^2-\frac{p}{2}|z|^2}\bigg) dA(z)\\
   = \frac{p}{2\pi} \int_{\CC} |f(z-z_0)|^p e^{-\frac{p}{2}|z-z_0|^2} dA(z)=\|f\|_p^p
   \end{align*} for all $1\leq p<\infty$ which also holds true for $p= \infty$.    This shows  that the operator is a linear  isometry and hence  satisfies the injectivity  condition   $ W_{(k_{z_0},\psi_{z_0})}^{-1} W_{(k_{z_0},\psi_{z_0})}= I.$
   On the other hand, for  each $f\in \mathcal{F}_{p}$
   \begin{align*}
   W_{(k_{z_0},\psi_{z_0})} W_{(k_{-z_0},\psi_{z_0}^{-1})} f(z)= k_{z_0}(z) k_{-z_0}(z-z_{0})  f(z)= f(z)
   \end{align*} which also shows that  $W_{(k_{z_0},\psi_{z_0})}W_{(k_{z_0},\psi_{z_0})}^{-1}= I$, and hence  the claim.\\
Next,   using $z_0= b/(1-a)$ and Lemma~\ref{lem0} for every  $f\in \mathcal{F}_{p}$ we compute
\begin{align*}
W_{(k_{-z_0},\psi_{z_0}^{-1})} W_{(u,\psi)} W_{(k_{z_0},\psi_{z_0})}  f(z)\quad \quad \quad \quad\quad \quad \quad \quad\quad \quad \quad \quad\\
= k_{-z_0}(z)u(\psi_{z_0}^{-1}(z)))k_{z_0}(\psi(\psi_{z_0}^{-1}(z)) f( \psi_{z_0}(\psi(\psi_{z_0}^{-1}(z))))\\
= k_{-z_0}(z)u(0) K_{-\overline{a}b}(z+z_0) k_{z_0}(az+b+az_0) f(az)= u(0)e^{\frac{|b|^2}{1-\overline{a}}}  C_{\Psi_0}  f(z)
\end{align*} where $C_{\psi_0}$ the composition operator induced by the symbol $\Psi_0(z)= az$.  This shows that $ W_{(u,\psi)} $ is similar to the  composition operator, up to a multiple,  $C_{\Psi_0}$. Thus, $\sigma ( W_{(u,\psi)})=  u(0)e^{\frac{|b|^2}{1-\overline{a}}}  \sigma (C_{\Psi_0})$.  Using the spectrum of $C_{\Psi_0}$ from Theorem~2.6 of \cite{TW2} and observing that  $(1-\overline{a})^{-1}= a/(1-a) $ when $|a|=1$ and $a\neq1$, we arrive at  the desired conclusion.

\noindent \emph{Case 2}. Let  $a=1$ and $b\neq 0$.  Applying Lemma~ \ref{lem0},
  \begin{align} \label{spec1}
  W_{u, \psi} = u \cdot  C_{\psi}  =u(0)K_{-b} C_{\psi} = u(0)e^{\frac{|b|^{2}}{2}} k_{-b} C_{\psi}  = u(0)e^{\frac{|b|^{2}}{2}}  W_{ ( k_{-b}, \psi ) }
  \end{align}
  The weighted composition operator  $ W_ { ( k_{-b}, \psi ) }    $ is unitary. Recall that the spectrum of a unitary operator lies on the unit circle $\mathbb{T}$.
%  So, the spectrum of $ W_ { ( k_{-b}, \psi ) }    $ is contained in the unit circle $\mathbb{T}$.
 We claim that the spectrum of  $ W_ { ( k_{-b}, \psi ) }    $ is $\mathbb{T}$.
 To prove the claim,  for any nonzero  $w\in \CC$ and $f\in \mathcal{F}_{p}$ we have
\begin{align*}
W_{(k_{-w},\psi_{w}^{-1})} W_{(k_{-b},\psi)} W_{(k_{w},\psi_{w})}  f(z)\quad \quad \quad \quad\quad \quad \quad \quad\quad \quad \quad \quad\quad \quad \quad \quad\quad \quad \quad\\
= k_{-w}(z) k_{-b}(z+w) k_{w}(z+b+w) f(z+b) \quad\quad \quad \quad \quad\quad \quad \quad\\
=k_{-b}(z)e^{2i \Im (\overline w b)}f(z+b)
= e^{2i \Im (\overline w b)}W_{(k_{-b},\psi)}  f(z)\quad \quad
\end{align*} which shows that $ W_{(k_{-b},\psi)}$ is similar to $e^{2i \Im (\overline w b)}W_{(k_{-b},\psi)}$ for any  $w\in \CC$.
 Since $b\neq 0,$ and $e^{2i \Im (\overline w b)}$ is unimodular, and the spectrum of a unitary operator lies on the unit circle $\mathbb{T}$, it follows that  the whole unit circle constitutes the spectrum.  \\
 Therefore, combining this with \eqref{spec1}
  \begin{align*}
  \sigma (W_ { ( u, \psi )}) = u(0)e^{\frac{|b|^{2}}{2}} \sigma\big(  W_ { ( k_{-b}, \psi ) }\big) = u(0)e^{\frac{|b|^{2}}{2}} \mathbb{T} =\left \{ z : |z|=  |u(0)| e^{\frac{|b|^{2}}{2}}      \right\}.
  \end{align*} and completes the proof.

\noindent\emph{Case 3}. Let  $a=1$ and $b= 0$. In this case, the operator  $W_{(u,\psi)}$ reduces to the multiplication operator $M_u$ where its spectrum has been  already  identified  in Lemma~2.3 of \cite{TM5}.

\section{Uniformly mean ergodic $W_{(u,\psi)}$  } \label{uniform}
Having  identified conditions under which  $W_{(u,\psi)}$  is power bounded, we next turn our attention to the mean and uniformly mean ergodic properties of  $W_{(u,\psi)}$ on  $\mathcal{F}_p$.  We may first
state the following result about compact weighted composition operators on the spaces.
\begin{proposition}
Let $1\leq p\leq \infty$ and $W_{(u,\psi)}$ be a compact power  bounded operator on  $\mathcal{F}_p$, and hence   $\psi(z)= az+b$ such that $|a|<1$. Let $u$ be  non-vanishing on $\CC$ .  Then  $W_{(u,\psi)}$ is uniformly  mean ergodic.
\end{proposition}
\begin{proof} A result of Yosida and Kakutani (\cite{YK} Theorem 4 and Corollary on page 204--205 and Theorem~2.8 in \cite{Kr}) implies that every compact power bounded operator on a Banach space is uniformly mean ergodic. Thus,  $W_{(u,\psi)}$  is uniformly mean ergodic.
\end{proof}
\begin{theorem}
                \label{thmm3}
                Let $1\leq p< \infty$ and $W_{(u,\psi)}$ be a compact power  bounded operator on  $\mathcal{F}_p$, and hence   $\psi(z)= az+b$ such that $|a|<1$. Let $u$ be  non-vanishing on $\CC$ such that $u( z_0)=1$ where $z_0= b/(1-a)$.  Then
                \begin{align} \label{the limit}
\lim_{n \rightarrow \infty} \left \| \frac{1}{n} \sum_{k=1}^{n} W_{(u,\psi)}^{k} - W_{(u_\infty, z_0)} \right\|=0, \ \  u_\infty(z)=\prod_{j=0}^\infty u(\psi^j(z)).
\end{align}
 \end{theorem}
                \begin{proof}
                To prove \eqref{the limit}, we argue as follows. First  observe  that $u_\infty$ is a well-defined product as  $W_{(u,\psi)}$ is power bounded, Lemma~\ref{lem2} and \eqref{global} imply
                \begin{align*}
                |u_\infty(z)|= \lim_{n\to \infty}|u_n(z)|\lesssim \lim_{n\to \infty} e^{\frac{1}{2}|z|^2 } \|u_n\|_p \lesssim e^{\frac{1}{2}|z|^2 }.
                                                \end{align*}
 For  each $f\in \mathcal{F}_p$    we claim that
     \begin{eqnarray}
     \label{claim}
\lim_{n \rightarrow \infty} \| W_{(u,\psi)}^{n}f - W_{(u_\infty, z_0)} f \|_{p}=0.
\end{eqnarray}Observe that \eqref{claim} implies  \eqref{the limit}. Thus, we proceed to prove the claim by considering
 two different cases.  Let $p<\infty$.  Since $W_{(u,\psi)}$ is power bounded, by Lemma~\ref{lem2}, $u_n$  belongs to $ \mathcal{F}_{p}$ for all $n\in \NN$.   On the other hand, using the representation of $u_n$ in \eqref{exp} and applying  Theorem~\ref{thm2} and Lemma~\ref{techlemma}

\begin{align*}
|u_{n}(z)|\lesssim  e^{|t_{n}z|+\left| \frac{a_{2}(1-a^{2n})}{1-a^{2}} \right| |z|^{2}} \leq e^{|t_{n}z|+|a_{2}|\frac{1-|a|^{2n}}{1-|a|^2} |z|^{2}}
\end{align*} where
$$
t_n= \frac{a_1(1-a^n)}{1-a}+   \frac{2a_2b}{1-a}\bigg(\frac{1-a^n}{1-a}-\frac{1-a^{2n}}{1-a^2} \bigg)+ \overline{\frac{(1-a^n)b}{1-\overline{a}}} a^n,
$$
 and letting $n\to \infty $
\begin{align*}
|u_{\infty}(z)|\lesssim  e^{\Big(\frac{2 |a_1|}{|1-a|}+\frac{2|a_2b|}{|1-a|}\big(\frac{2}{|1-a|}+\frac{2}{|1-a^2|}\big)\Big)|z|+\frac{|a_2|}{|1-a^2|}|z|^2 }.
\end{align*}  The compactness condition  $|a_2| <\frac{1-|a|^2}{2}$  implies $\frac{|a_2|}{|1-a^2|}<\frac{1}{2} $,
which shows that   $u_{\infty} \in \mathcal{F}_{p}$.

Moreover, by continuity,
\begin{eqnarray*}
\lim_{n \rightarrow \infty} \left| W_{(u,\psi)}^{n}f (z) - W_{(u_\infty, z_0)} f(z)  \right |=0,
\end{eqnarray*}  and  since $W_{(u,\psi)}$ is power bounded, there is a constant $\alpha > 0$ such that
 for every $n \in \NN$
 \begin{eqnarray*}
\int_{\CC}  \left| W_{(u,\psi)}^{n}f (z) \right|^{p}  e^{-\frac{p}{2}|z|^2} dA(z) \leq \int_{\CC} \alpha^{p} |f(z)|^{p} e^{-\frac{p}{2}|z|^2} dA(z).
\end{eqnarray*}

Consequently, we  have
\begin{eqnarray*}
%\label{dom}
\int_{\CC}  \left|  W_{(u,\psi)}^{n}f(z) -W_{(u_\infty, z_0)} f(z) \right|^{p}  e^{-\frac{p}{2}|z|^2} dA(z)\quad \quad \quad \quad \quad \quad \quad \quad \quad \quad\nonumber\\
 \leq \int_{\CC} 2^{p}  \bigg(\alpha^{p} \left| f(z) \right|^{p} + \left|W_{(u_\infty, z_0 )}f(z)  \right|^{p} \bigg)  e^{-\frac{p}{2}|z|^2}  dA(z) < \infty.
\end{eqnarray*}
\noindent Applying Lebesgue dominated convergent theorem  on the sequence
\begin{eqnarray}
g_{n}(z): = \left|  W_{(u,\psi)}^{n}f(z) -W_{(u_\infty, z_0)} f(z) \right|^{p}  e^{-\frac{p}{2}|z|^2}, \nonumber
\end{eqnarray}
we get that
\begin{eqnarray}
\lim_{n \rightarrow \infty} \| W_{(u,\psi)}^{n}f-W_{(u_\infty, z_0} f\|^{p}_{p}
 = \frac{p}{2\pi}\lim_{n \rightarrow \infty} \int_{\CC} g_{n}(z)dA(z)
=  \frac{p}{2\pi}\int_{\CC} \lim_{n \rightarrow \infty}g_{n}(z)dA(z)=0\nonumber
\end{eqnarray} as claimed.
\end{proof}

\begin{theorem}
                \label{thmm33}
                Let  $W_{(u,\psi)}$ be a compact power  bounded operator on  $\mathcal{F}_\infty$, and hence   $\psi(z)= az+b$ such that $|a|<1$. Let $u$ be  non-vanishing on $\CC$ such that $u( z_0)=1$ where $z_0= b/(1-a)$.  Then
                \begin{align} \label{the limitt}
\lim_{n \rightarrow \infty} \left \| \frac{1}{n} \sum_{k=1}^{n} W_{(u,\psi)}^{k} - W_{(u_\infty,z_0)} \right\|=0, \ \  u_\infty(z)=\prod_{j=0}^\infty u(\psi^j(z)).
\end{align}
 \end{theorem}
 \begin{proof}
  For this,  we consider the subspace  $ \mathcal{F}_{0}$ of $ \mathcal{F}_{\infty}$ defined  by
\begin{align*}
%\label{zero}
\mathcal{F}_{0}= \lbrace f \in \mathcal{F}_{\infty}: \lim_{|z|\rightarrow \infty}|f(z)|e^{-\frac{1}{2}|z|^2}=0\rbrace.
\end{align*}
This subspace is closed in  $ \mathcal{F}_{\infty}$ and it contains polynomials. Moreover, the polynomials are dense in  $ \mathcal{F}_{0}$, and  $ \mathcal{F}_{\infty}$ is canonically isomorphic to the bidual of $ \mathcal{F}_{0}$; see for details in \cite{BS}. We proceed to show  first that \eqref{claim} holds for each $f \in \mathcal{F}_{0}$.
It is easy to see that as $n \to \infty$
\begin{align*}
\psi^{n}(z)= a^{n} z + z_0(1-a^{n}) \rightarrow z_0
\end{align*}
uniformly on compact subsets of $\CC$.

Next, we show that $u_n\to u_{\infty}$ uniformly on compact subset of $\CC$ also. Since  $u_{\infty} \in \mathcal{F}_{1}$,
\begin{small}
\begin{eqnarray}
\int_{\CC}  \left|u_{n}(z) -  u_{\infty}(z)\right|  e^{-\frac{1}{2}|z|^2} dA(z) \leq \int_{\CC}2 \left( \alpha +|u_{\infty}(z)| \right) e^{-\frac{1}{2}|z|^2} dA(z) < \infty.
\label{lebegue two}
\end{eqnarray}
\end{small}
Applying Lebesgue dominated convergence theorem, for $z \in K,$ where $K$ is compact subset of $\CC,$
\begin{eqnarray}
\lim_{n \to \infty} |u_{n}(z)-u_{\infty}(z)| \leq \lim_{n \to \infty} e^{\frac{1}{2}|z|^{2}}\|u_{n}-u_{\infty}\|_{1}
\leq \left(\max_{z \in K}e^{\frac{1}{2}|z|^{2}}\right)\lim_{n \to \infty}\|u_{n}-u_{\infty}\|_{1} \nonumber\\
 = \left(\max_{z \in K}e^{\frac{1}{2}|z|^{2}}\right)\lim_{n \to \infty}\int_{\CC}|u_{n}(w)-u_{\infty}(w)|e^{-\frac{1}{2}|w|^{2}}dA(w)=0. \nonumber
\end{eqnarray}
From this we have
\begin{align*}
u_n(z)f(\psi^{n}(z))\rightarrow  u_\infty(z)f ( z_0 )
\end{align*}
uniformly on the compact subsets of $\CC$. That is, for each compact set $K$ in $\CC$,
\begin{align} \label{six}
 \sup_{z \in K} \left|u_n(z)f(\psi^{n}(z))- u_\infty(z)f ( z_0 ) \right| \rightarrow 0
\end{align}
as $n \rightarrow \infty$.

 Next, with  $f\in \mathcal{F}_0$ and each  $n$,
 \begin{align*}
 \lim_{|z|\to \infty}  |W_{(u,\psi)}^n f(z)|e^{-\frac{1}{2}|z|^2}= \lim_{|z|\to \infty}|u_n(z)|\bigg|f\Big(a^n z+  z_0(1-a^n)\Big)\bigg|^2 e^{-\frac{1}{2}|z|^2}\quad\quad \quad  \\
 \leq \sup_{z\in \CC}\bigg(|u_n(z)|e^{\frac{1}{2}\big(|a^n z+ z_0(1-a^n)|^2-|z|^2\big)}\bigg)\quad\quad \quad\quad\\
 \times \lim_{|z|\to \infty} \bigg|f\Big(a^n z+ v\Big)\bigg|^2 e^{-\frac{1}{2}|a^n z+ z_0(1-a^n)|^2}\quad\quad\quad \quad\\
 \lesssim \lim_{|z|\to \infty} \bigg|f\Big(a^n z+ z_0(1-a^n)\Big)\bigg|^2 e^{-\frac{1}{2}|a^n z+ v|^2}= 0.\quad\quad \quad
 \end{align*} Note that the last inequality above holds since $W_{(u,\psi)}$ is power bounded the supremum above is uniformly bounded. That is
 \begin{align*}
 \sup_n\sup_{z\in \CC}\Big(|u_n(z)|e^{\frac{1}{2}\big(|a^n z+ z_0(1-a^n)|^2-|z|^2\big)}\Big)\lesssim |u(z_0)|^n \leq 1.
 \end{align*}
 Furthermore,   $u_{\infty}\in \mathcal{F}_{0} $  since it belongs to $\mathcal{F}_{p}$  for all $p\leq \infty.$

Now given $\varepsilon >0,   f \in \mathcal{F}_{0}$ and  since $u_\infty \in \mathcal{F}_{0}$, we can  find $r_0 > 0$ such that $|W_{(u,\psi)}^n f(z)|e^{-\frac{1}{2}|z|^2} < \varepsilon/2$ and $|u_\infty(z)||f(z_0)| e^{-\frac{1}{2}|z|^2} < \varepsilon/2$ whenever $|z|>r_0$. Then, for each  $|z| > r_0$ and $n\in  \NN$, we have
\begin{align*}
\left| W_{(u,\psi)}^n f(z) - u_\infty(z)f(z_0) \right|e^{-\frac{1}{2}|z|^2} \leq \big| W_{(u,\psi)}^n f(z)\big| e^{-\frac{1}{2}|z|^2}\quad \quad \quad \quad \\
+ \bigg|u_\infty(z)f(z_0)\bigg| e^{-\frac{1}{2}|z|^2} <\varepsilon.
\end{align*}
We apply (\ref{six}) to the compact set $K_0= \{z \in \CC :  \ |z| \leq r_0  \}$ to find $n_0$ such that if $z \in K_0$ and $n \geq n_0$ we have
\begin{align*}
\left|u_n(z)f(\psi^{n}(z))-u_\infty(z)f ( z_0) \right| < \frac{\varepsilon}{2S},
\end{align*}
with $ S:=\max_{z \in K_0} e^{-\frac{1}{2}|z|^2}$. If $n \geq n_0$ and $z \in \CC$, we have
\begin{align*}
\left|W_{(u,\psi)}^nf(z) - u_\infty(z)f ( z_0)  \right|e^{-\frac{1}{2}|z|^2} < \varepsilon.
\end{align*}
Thus \begin{align*}
\lim_{n \rightarrow \infty} \| W_{(u,\psi)}^nf - u_\infty f ( z_0) \|_{ \infty}=0.
\end{align*}
Next,  we show \eqref{the limitt} for $p=\infty.$ Since  $\mathcal{F}_{\infty}$ is canonically isomorphic  to the bidual of $\mathcal{F}_{0}, $ and the bitranspose operator $W^{''}_{(u.\psi)}$ of
$W_{(u,\psi)}: \mathcal{F}_{0} \rightarrow \mathcal{F}_{0} $ coincides  with  weighted composition operator $W_{(u,\psi)}: \mathcal{F}_{\infty} \rightarrow \mathcal{F}_{\infty}$, the conclusion follows from the well-known fact that
$\|T \|= \|T^{'} \|= \| T^{''}\|$ for any bounded operator $T$ on a Banach space.
\end{proof}
The preceding results assure that $W_{(u,\psi)}$ with  $\  \psi(z)= az+b, \; |a|<1$ is always uniformly mean ergodic whenever it is compact and  power bounded.  Now,  we consider  the case
when $\psi(z)= az+b$ and $|a|=1$. Note that power boundedness in this case  implies that either $|u(0)|=e^{-\frac{|b|^2}{2}}$ or $|u(0)|<e^{-\frac{|b|^2}{2}}$.   In 1939,  Lorch \cite{Lorch} proved that every power bounded operator on a reflexive Banach space is mean ergodic. The same result was latter obtained in  reflexive Frechet spaces \cite{ABJ}. Accordingly, as  the spaces $\mathcal{F}_p$ are reflexive  for all $1<p<\infty$,  every power bounded $W_{(u,\psi)}$ is mean ergodic. Thus, for such spaces we will consider conditions under which the ergodicity becomes uniform.
                \begin{theorem}
                \label{thm4}
                \begin{enumerate}
                \item
                Let $1\leq p\leq  \infty$ , $\psi(z)= az+b$ with $|a|= 1$ and             $|u(0)|< e^{-\frac{|b|^{2}}{2}}$. Then   $W_{(u,\psi)}$ is uniformly mean ergodic on $ \mathcal{F}_p$, and
  \begin{align*}
\lim_{n\to \infty}\bigg\| \frac{1}{n} \sum_{k=0}^n W_{(u,\psi)}^{k}  \bigg\|=0.
  \end{align*}
  \item Let $1\leq p\leq \infty$,and   $\psi(z)= az$ with $ |a|=1$. If  both $u(0)$ and $a$  are roots of unity, then $W_{(u,\psi)}$ is uniformly mean ergodic on $ \mathcal{F}_p$.
\end{enumerate}
    \end{theorem}
     By a result in \cite{MMJ}, the   composition operator $C_\psi$ is not uniformly mean ergodic on $ \mathcal{F}_\infty$ whenever  $|a|= 1$. Now the weight function $u$ makes it possible to enrich  uniformity by taking the value  $|u(0)|$ smaller.\\
      \begin{proof}
   (i)  Applying the assumption along with \eqref{it}
                                 \begin{align*}
  \lim_{n\to \infty}\bigg\| \frac{1}{n} \sum_{k=0}^n W_{(u,\psi)}^{k}  \bigg\|\leq  \lim_{n\to \infty} \frac{1}{n} \sum_{k=0}^n\bigg\| W_{(u,\psi)}^{k}  \bigg\| \lesssim \lim_{n\to \infty} \frac{1}{n} \sum_{k=0}^n\left(|u(0)| e^{\frac{|b|^{2}}{2}} \right)^{k}\\
  \leq\lim_{n\to \infty}\frac{2n^{-1}}{1-|u(0)| e^{\frac{|b|^{2}}{2}}}= 0 \quad \quad
  \end{align*} as claimed.\\
(ii) By assumption there exist  numbers $m, N \in \NN $ such that $a^{N}=1=u(0)^m$. Consider the smallest positive integer $N_0\leq mN$ such that $a^{N_0}=u(0)^{N_0}=1 $.  In this case the sequence $W_{u,\psi}^{n}$ is periodic with period $N_0$. Any  $n \in \NN $ can be written  in the form of $n = N_0l+j $ for some $ l \in  \NN $ and  $j=0,1,2,...,N_0-1.$ Thus
\begin{eqnarray*}
\lim_{n \rightarrow \infty}  \left \| \frac{1}{n} \sum_{k=1}^{n} W_{(u,\psi)}^{k}  - \frac{1}{N_0} \sum_{k=1}^{N_0} W_{(u, \psi)}^{k} \right\|= \lim_{l \rightarrow \infty} \frac{1}{(N_0l +j)} \left \|  \sum_{k=1}^{j} W_{(u, \psi)}^{k}  - \frac{j}{N_0} \sum_{k=1}^{N_0} W_{(u,\psi)}^{k} \right\|\nonumber \\
  \leq \lim_{l \rightarrow \infty}  \frac{1}{(N_0l +j)}\left( \sum_{k=1}^{j}   \left \|  W_{(u,\psi)}^{k}  \right\|  + \frac{j}{N_0} \sum_{k=1}^{N_0} \left \|   W_{(u,\psi)}^{k} \right\| \right)  = 0 \nonumber
\end{eqnarray*} and completes the proof.
\end{proof}

Next,  we consider the cases when the uniform ergodicity  fails. We may first recall that a Banach space $X$ is a Grothendieck space if every sequence $(x_{n})$ in $X'$  which is convergent to 0 for the weak topology  $ \sigma (X^{'} , X) $ is also convergent to 0 for the weak topology $ \sigma (X^{'} , X^{''}) $. The space $X$ has the Dunford-Pettis property if for any sequence $(x_{n})$ in $X$ which is convergent to $0$ for the weak topology  $ \sigma (X , X^{'}) $ and any sequence $(x_{n}^{'})$ in $X^{'}$ which is convergent to $0$ for the weak topology  $ \sigma (X^{'} , X^{''}) $ one gets $  \lim_{n \rightarrow \infty} x_{n}^{'}(x_{n})=0$.
The spaces  $ \ell^{\infty}$ or $H^{\infty}(\mathbb{D})$ are examples of Grothendieck spaces with Dunford-Pettis property \cite{HP}.
\begin{theorem}
\label{thm5}
 Let $1\leq p\leq \infty$ and $W_{(u,\psi)}$ is power  bounded on $\mathcal{F}_{p}$ with  $\psi(z)= az$, $ |a|=1$. Then $W_{(u,\psi)}$ is
 \begin{enumerate}
 \item mean ergodic on $ \mathcal{F}_p$ for all  $1\leq p < \infty$ if   $1\neq u(0) a^m$ for all $m\in \NN_0$, and  for each $f\in \mathcal{F}_p$
  \begin{align*}
  \lim_{n\to \infty}\bigg\| \frac{1}{n} \sum_{k=1}^n W_{(u, \psi)}^{k} f \bigg\|_p=0,
  \end{align*}
  \item   mean ergodic on $ \mathcal{F}_p$ for all  $1\leq p< \infty$ if   $u(0)=1$ and $a$ is  not root of unity, and for each $f\in \mathcal{F}_p$
  \begin{align*}
  \lim_{n\to \infty}\bigg\| \frac{1}{n} \sum_{k=1}^n W_{(u, \psi)}^{k} f-f(0) \bigg\|_p=0,
  \end{align*}
\item not uniformly mean ergodic on $\mathcal{F}_{p}$ for all $1\leq p\leq \infty$, and not mean ergodic on $ \mathcal{F}_\infty$  whenever    $1\neq u(0) a^m$ for all $m\in \NN_0$.
\end{enumerate}
  \end{theorem}
\begin{proof}
(i) and (ii).  We first check when $f$ belongs to the set of monomials.
 If $f= 1$, then the result holds trivially. Thus,   for  $z^m, m\geq 1$ Assume that    $1\neq u(0) a^m$ for all $m\in \NN$.
\begin{align*}
                \bigg\| \frac{1}{n} \sum_{k=1}^n W_{(u,\psi)}^{k} z^m \bigg\|_p=  \bigg\| \frac{1}{n} \sum_{k=1}^n  a^{mk} u_k z^m  \bigg\|_p.
                \end{align*}
  Since $a^{k}u(0)\neq 1$ for each $k \in \NN$ and  $u_k(z)= u(0)^k $, or $u(0)=1$ and $a$ is not root of unity
                  \begin{align*}
             \bigg\| \frac{1}{n} \sum_{k=1}^n  a^{mk} u_k z^m  \bigg\|_p   = \bigg\| \frac{1}{n} \sum_{k=1}^n  u(0)^k a^{mk}  z^m  \bigg\|_p
                             = \Big\| \frac{z^m }{n} \frac{a^m(1- u(0)^na^{mn})}{1- a^{m}u(0)}     \Big\|_p\\
                              \leq   \frac{2\big\|z^m\big\|_p }{n |1- a^{m}u(0)|}\to 0 \quad \quad \quad \quad \quad \quad \quad \quad \quad \quad
                \end{align*} as $n\to \infty$.
Since the set of polynomials is dense in $ \mathcal{F}_{p}$  and $W_{(u,\psi)}$ is power bounded   on  $ \mathcal{F}_{p} $ by Theorem~\ref{thm1}, we have  the result (see e.g.\ Lemma 2.1 in \cite{Beltran}).\\
(iii) Assume now on the contrary  that $W_{u,\psi}$ is uniformly mean ergodic.   Then, by the classical result of Lin \cite{Lin}), $ \text{Im}(I-W_{(u,\psi)})$ is closed where
 $ \text{Im}(I-W_{(u,\psi)})$ denotes the range of $I-W_{(u,\psi)}$. Moreover, by  a result of Yosida \cite{Y} it follows that $Ker T_0 = \overline{Im(I-W_{(u, \psi)}})$ , where $ T_0= \lim_n T_n$ and $ T_n= \frac{1}{n} \sum_{k=1}^{n} W_{(u,\psi)}^{k}  $.  Hence
\begin{align}
\label{closed}
\text{Im}(I-W_{(u,\psi)})= \overline{\text{Im}(I-W_{(u,\psi)})}=Ker T_0= \bigg\{ f \in \mathcal{F}_{p}: \lim_{n \to \infty} \frac{1}{n} \sum_{k=1}^{n} W^{k}_{(u,\psi))}f = 0 \bigg\}\nonumber\\
=\bigg\{ f \in \mathcal{F}_{p} : \lim_{n \to \infty}\frac{1}{n} \sum_{k=1}^{n} u(0)^kf(\psi^k) = 0 \bigg\}.
\end{align}  where the last equality follows after an application of Lemma~\ref{lem0}. Furthermore,
 \begin{align}
 \label{direct}\mathcal{F}_{p}= \text{Im}(I-W_{(u,\psi)}) \bigoplus \text{Ker}(I-W_{(u,\psi)}).\end{align}  We claim that  $\text{Ker}(I-W_{(u,\psi)})$
 contains only the constant functions. Indeed, if $f$ belongs to it, then
  since $u(0) f(az)=f(z) $ for each $ z \in  \CC $, we have
$ u(0) a^{n} f^{(n)}(az) = f^{(n)}(z), $ and hence $ u(0) a^{n} f^{(n)}(0) = f^{(n)}(0). $ This yields $  f^{(n)}(0) = 0 $ for every $ n \in \NN $ since  $u(0)  a^{n} \neq 1 $. If we define $ g(z) = f(z) - f(0),  $ it follows that $ g^{(n)}(0)=0 $  for every $n=0,1,2,...  $. Hence $ g(z)= 0 $ for every $ z \in \CC. $
Therefore,  $ f(z)=f(0)$ for every  $ z \in \CC. $

Next, we show that the constant functions belong to the set in \eqref{closed} and contradicts \eqref{direct}. Thus, if  $h= \alpha$ is a non-zero constant function, then  $  W^{k}_{(u,\psi))}h(z)= u(0)^k h(\psi^k(z))=  u(0)^k\alpha $ for every  $ z \in \CC. $  Thus
\begin{align*}
\bigg\|\frac{1}{n} \sum_{k=1}^{n} h(\psi^k)\bigg\|_p\leq   \frac{\|\alpha\|_p}{n} \bigg|\sum_{k=1}^{n} u(0)^k\bigg|=  \frac{ \|\alpha\|_p}{n}\bigg| \frac{u(0)-u(0)^{n+1}}{1-u(0)}\bigg|\\
\leq   \frac{2 \|\alpha\|_p}{n|1-u(0)|}\to 0 \ \ \text{as}\ \  n \to \infty.
\end{align*}
The case when  $1= u(0)$ and $ a$ is not a root of unity, follows as  a special case from Corollary~\ref{cors} below.

It remains to show that $W_{(u,\psi)}$ is not mean ergodic on $\mathcal{F}_\infty$. Theorem 1.1 in \cite{WL} implies that
$\mathcal{F}_\infty$ is isomorphic to $ \ell^{\infty}$ or $H^{\infty}(\mathbb{D})$. Hence  $\mathcal{F}_\infty$ is Grothendieck spaces with Dunford-Pettis property.  On the  other hand, by a result of Lotz \cite{HP2} every power bounded mean ergodic operator on a Grothendieck Banach space with the Dunford-Pettis property is uniformly mean ergodic. Therefore,  $W_{(u,\psi)}$ is not mean ergodic   on $\mathcal{F}_{\infty}$ .
\end{proof}
We remark that when $|a|=1$,  the operators are isometric  bijective with
$W_{(u,\psi)}^{-1}=  W_{(v,\psi^{-1})} $ where $v(z):= \overline{u(0)}K_{b}(z)$. This can be seen as for every  $f\in \mathcal{F}_{p}$
\begin{align*}
   \|W_{(u,\psi)}f\|_p^p=  \frac{p}{2\pi} |u(0)|^p \int_{\CC} |K_{-\overline{a}b}(z)|^p |f(az+b)|^p e^{-\frac{p}{2}|z|^2}dA(z)\\
   =\frac{p}{2\pi} |u(0)|^p \int_{\CC}  |f(az+b)|^p e^{-\frac{p}{2}|az+b|^2} \bigg( |K_{-\overline{a}b}(z)|^p e^{\frac{p}{2}|az+b|^2-\frac{p}{2}|z|^2}\bigg) dA(z)\\
   =\frac{1}{|a|^2} |u(0)|^p e^{\frac{p|b|^2}{2}}\|f\|_p^p=\|f\|_p^p
   \end{align*} for all $1\leq p<\infty$ which also holds true for $p= \infty$.  This shows  that the operator is a linear  isometry and hence  satisfies the injectivity  condition   $ W_{(u,\psi)}^{-1} W_{(u,\psi)}= I.$
   On the other hand, for  each $f\in \mathcal{F}_{p}$
   \begin{align*}
   W_{(u,\psi)} W_{(v,\psi^{-1})} f(z)= u(z)\overline{u(0)}K_{b}(\psi(z)). f(\psi^{-1}(\psi(z)))\quad \quad\quad \quad \quad \quad\\
   = u(z)\overline{u(0)}K_b(az+b) f(z)
   = |u(0)|^2K_{-\overline{a}b}(z) K_b(az+b)  f(z)= f(z)
   \end{align*} which also shows that  $W_{(u,\psi)} W_{(u,\psi)}^{-1}= I$. As shown below and Theorem~\ref{spectrum}, the spectrum of  some of these class of operators are contained in the unit circle. \\
The uniformly mean ergodic results in part (ii) of Theorem~\ref{thm4} and Theorem~\ref{thm5} deal with when  $\psi(z)= az+b$ form   with  $|a|=1$ and    $ b=0$.  The case for $b\neq 0$ is our next point of interest.
\begin{corollary}\label{cors}
 Let $1\leq p\leq \infty$,    $\psi(z)= az+b$, $|a|=1$ and $a$ is not a root of unity,   $   u\left(\frac{b}{1-a} \right) = 1 $ and hence $|u(0)| = e^{ - \frac{|b|^{2}}{2}}$.  Then $W_{(u,\psi)}$ can not be  uniformly mean ergodic on $\mathcal{F}_{p}$.
 \end{corollary}
\begin{proof}
 First observe that in this case Theorem~\ref{spectrum} and since   for $|a|=1$ and $a\neq 1$,
  \begin{align*}
\big|u(0)e^{\frac{a |b^2|}{a-1}}a^m\big|= |u(0)|  e^{|b|^2\Re\big(\frac{a }{a-1}\big)}= |u(0)|  e^{|b|^2\Re\big(\frac{a(\overline{a}-1) }{(a-1)(\overline{a}-1)}\big)}
= |u(0)|  e^{\frac{|b^2|}{2}}= 1,
 \end{align*}the spectrum $\sigma(W_{(u,\psi)})$ is contained in the unit circle $\mathbb{T}$. Furthermore as $a$ is not a root of unity , it follows that $1$ is an accumulation point of the spectrum of $W_{(u,\psi)}$.  Moreover 1 is in the spectrum of $W_{(u,\psi)}$ since $ u(0)e^{\frac{a |b^2|}{a-1}} a^{0} =  u\left(\frac{b}{1-a} \right)  a^{0} = 1$.  Then an application of   Theorem~3.16 of \cite{Dunford} gives the conclusion.
 \end{proof}
\subsection{The multplication operator}\label{mult1}
We now conclude the section by  specializing the main results made in the above section to the multiplication  operator $M_u$  acting on Fock spaces.
Note that from Lemma~2.3 in \cite{TM5}, it is known that the operator $M_u$  is bounded on $\mathcal{F}_p$  if and only if $u$ is a constant function. The same conclusion can be also  easily drawn by applying  the condition in  \eqref{bounded} along with  Liouville's Theorem.
  \begin{corollary}\label{corM}
  Let $1\leq p\leq \infty$ and $u\in \mathcal{H}(\CC)$ such that $M_u$ is bounded on $\mathcal{F}_p$. Then the  following statements  are equivalent.
                                  \begin{enumerate}
                                   \item  $M_u$  is power bounded  on $ \mathcal{F}_p$;
                                   \item   $|u(0)|\leq 1$;
                                  \item $M_u$  is mean ergodic  on $ \mathcal{F}_p$;
                                  \item   $M_u$ is uniformly mean ergodic on $ \mathcal{F}_p$.
                                  \end{enumerate}
                                 \end{corollary}
                                 \begin{proof}
                             The equivalency of  (i) and  (ii) is an immediate deduction from Theorem~\ref{thm1}. Thus, we shall show  (ii) $\Rightarrow$ (iii), (iii)$\Rightarrow$ (iv), and (iv) $\Rightarrow$ (i).  For the first,   simplifying the  proof of  Theorem~\ref{thm4} part (i) for the case  $b= 0$ and  $a=1$, we get $W_{(u,\psi)}^{k}= M_u^k$ and
                               \begin{align}
                               \label{mult}
                \bigg\| \frac{1}{n} \sum_{k=1}^nM_u^k f \bigg\|_p=  \bigg\| \frac{ f}{n} \sum_{k=1}^n   u_k(0)  \bigg\|_p = \bigg\|  \frac{ f}{n} \sum_{k=1}^n  u(0)^k  \bigg\|_p  = \big\| f  \big\|_p\bigg| \frac{1}{n} \sum_{k=1}^n    u(0)^k  \bigg|.
      \end{align}
      Consider first the case when  $u(0)\neq1$ and  $|u(0)|\leq1$. Then
      \begin{align*}
            \frac{\big\| f \big\|_p}{n} \Big| \sum_{k=1}^n    u(0)^k  \Big|=  \frac{\big\|f  \big\|_p}{n} \frac{|u(0)||1-u(0)^n|}{|1-u(0)|}
      \leq  \frac{2\big\|f  \big\|_p}{n|1-u(0)|}  \to 0
      \end{align*} as $n\to \infty$. Thus, $\frac{1}{n} \sum_{k=1}^nM_u^k $ converges pointwise to zero.
     If  $u(0)=1$,  then $M_u$ reduces to the identity map and the assertion follows trivially.

                      \noindent
      Next, we show that (iii) implies (iv). That is the above convergence is uniform on the operator norm. Now the assumption implies that $ \frac{M_u^n f}{n} \to 0$ as $n\to \infty $ for all $f \in  \mathcal{F}_p$.  In particular for $f=\textbf{1}$, the statement
      $\frac{M_u^n\textbf{ 1}}{n} = \frac{u(0)^n}{n} \to 0$ holds  only if  $|u(0)|\leq 1$. Now, for $u(0)=1$, the operator  reduces again to the identity map. Thus, we consider the case  when $u(0)\neq1$, and argue
      \begin{align*}
       \Big\| \frac{1}{n} \sum_{k=1}^nM_u^k  \Big\|=  \sup_{\|f\|_p=1}\Big\| \frac{1}{n} \sum_{k=1}^nM_u^k f \Big\|_p= \sup_{\|f\|_p=1} \Big\| \frac{1}{n} \sum_{k=1}^n u(0)^k f \Big\|_p\\
       \leq \frac{1}{n} \sum_{k=1}^n |u(0)|^k=   \frac{|u(0)||1-|u(0)|^n|}{n(1-|u(0)|)}
      \leq  \frac{2}{n|1-u(0)|}  \to 0,  \ n\to \infty.
      \end{align*}

    \noindent   Now assume that (iv) holds. Then $\|M_u^n\|/n \to 0$ as $n\to \infty$. On the other hand, from \eqref{it}  we get  $\|M_u^n\|= |u(0)|^n$ which implies that    $\|M_u^n\|/n \to 0$  only  when   $|u(0)|\leq 1$. Therefore, by Theorem~\ref{thm1}, the operator is power bounded.
                 \end{proof}
               \subparagraph*{\textbf{Acknowledgement}}
  The authors would like to  thank  Professor Jos\'e Bonet for  useful discussions on the subject. We would  also  like  to  thank  the anonymous reviewers    for  useful   comments  that  helped  us   improve  the  manuscript.

\end{document}